\documentclass[12pt]{article}
\usepackage[cp1251]{inputenc}
\usepackage{amssymb}
\usepackage{amsfonts}
\usepackage{amsthm}

\usepackage{array}
\makeatletter
\@addtoreset{equation}{section}
\makeatother

\theoremstyle{plain} \newtheorem{Th}{Theorem}[section]
\theoremstyle{definition}
\newtheorem{lemma}[Th]{Lemma}
\newtheorem{Def}[Th]{Definition}
\newtheorem{remark}[Th]{Remark}
\newtheorem{Cor}[Th]{Corollary}

\begin{document}

\sloppy
\binoppenalty=10000
\relpenalty=10000

\title{On the two symmetries in the theory of $m$-Hessian operators\footnote{The paper was supported by the RFBR grants 15-01-07650, 15-31-20600.}}

\author{Ivochkina N. M.\footnote{St.-Petersburg State University,
Russia, 199034, St.Petersburg, Universitetskaya nab., 7-9. E-mail: ninaiv@NI1570.spb.edu.}, Filimonenkova N. V.\footnote{Peter the Great Saint-Petersburg Polytechnic University, Russia, 195251, St.Petersburg, Polytechnicheskaya, 29. E-mail: nf33@yandex.ru.}}

\date{Dedicated to the memory of Marek Burnat}

\maketitle

\begin{abstract}
We show that the modern theory of fully nonlinear operators had been started by the skew symmetry of minors in cooperation with the symmetry of symmetric functions. The paper presents  some consequences of this interaction for the $m$-Hessian operators. One of them is a setting of the isoperimetric variational problem for Hessian integral. The $m$-admissible minimizer is found, what brings out a new simple proof of the well known Poincare - type inequalities for Hessian integrals. Also a new set of inequalities, generated by a special finite set of functions, is found.
\end{abstract}

\section{Introduction}

The modern theory of fully nonlinear second order partial differential equations counts more than 35 years and  has been started by the papers
\cite{Ev82}, \cite{Kr83},
where the a priori estimates of Holder constants for the second derivatives of solutions have been constructed. It reduced the problem of the classic solvability of the Dirichlet problem for fully nonlinear second-order partial differential equations to construction of the a priori estimate of solutions in $C^2$. An attempt to give general description of admitting this estimate fully nonlinear operators may be found in
\cite{CNS85}, \cite{CNS88}, \cite{Kr83}.

There are other trends in this theory. One of them is to extend some known in the theory of the linear elliptic operators qualitative results to fully nonlinear operators. The first examples of such pattern are the embedding-type theorems for introduced in the papers
\cite{Ts}, \cite{W94}, \cite{TW98}
Hessian integrals.
Discussion on some other inherited from the linear case problems may be found, for instance, in the recent papers
\cite{V}, \cite{DPC}
and many others.

On the other hand, there are developments, which have no analogs in the linear theory, and these are of interest in our paper. It singles out the fully nonlinear operators of very special structure.  A classic representative of this kind is the Monge -- Ampere operator
$$\det u_{xx},\quad u\in C^2(\Omega),\quad\Omega\subset\mathbb R^n,$$
where $u_{xx}$ is the Hesse matrix of $u$.
Up to 1970 investigation of the Monge -- Ampere equation had been performed in the frames of differential geometry (see
\cite{P75}
and references therein). From 1975 the the Dirichlet problem for Monge -- Ampere equations became a model to modify methods developed in the theory of linear second-order partial differential equations to the fully nonlinear equations. In particular, it became the basis for the study of $m$-Hessian operators:
\begin{equation}T_m[u]=T_m(u_{xx}),\quad 0\leqslant m\leqslant n.\label{mh}\end{equation}
Here $T_0(S)\equiv1$, $T_m(S)$ is the $m$-trace of symmetric matrix $S$, that is the sum of all $m$-order principal minors. The set of operators (\ref{mh}) includes Laplace and Monge -- Ampere operators, $m=1$, $m=n$ respectively.

The $m$-Hessian operator is $m$-homogeneous and has two kinds of symmetries. The first is the orthogonal invariance of $m$-traces. Namely, if  $B$ is $n\times n$ an orthogonal matrix, then
\begin{equation}T_m(S)=T_m(BSB^T),\quad BB^T=Id.\label{ort}\end{equation}
Such symmetry admits a substitute of the $m$-traces of symmetric matrix by the elementary symmetric functions of the order $m$ of its eigenvalues $\lambda(S)$:
$$T_m(S)=S_m(\lambda(S))=\sum_{i_1<i_2<\ldots<i_m}\lambda_{i_1}\lambda_{i_2}\ldots\lambda_{i_m}.$$
It follows from the papers
\cite{CNS85}, \cite{T95}
that such symmetry suffices a classic solvability of the Dirichlet problem for $m$-Hessian equations. May be this is the reason that up to now the majority of scholars prefer to write $m$-Hessian operators (\ref{mh}) in terms of the eigenvalues of the Hesse matrix $D^2u=u_{xx}$:
\begin{equation} T_m[u]=S_m(\lambda[D^2u]).\label{sim}\end{equation}

The orthogonal invariance is well known type of symmetry of $m$-Hessian operators but in this paper we focus on the second type of symmetry, which we call a skew symmetry. In mid 70-$th$ this symmetry unnamed was discovered and investigated in quite different areas of mathematics. It brought out new nonlinear differential operators and mathematical models.

In Section 2 of this paper we give a brief of this story and show that the skew symmetric operators are divergence free, if homogeneous, generate exterior $n$-forms, etc. In fact, all this is the straightforward consequence of the skew symmetry of minors and that is why we say about skew symmetric functions and operators. In this paper we expose some well known for the set of $m$-Hessian operators relations as the consequences of this type of symmetry.

The approach of Section 2 makes reasonable to interpret Hessian integrals
$$I_m[u]:=\int_\Omega-uT_m[u]dx,\quad m=1,2,\ldots,n,$$
as a collection of new type of volumes  related to a bounded domain $\Omega\subset\mathbb R^n$ and functional sets
$$\left\{u\in C^2(\Omega): \;T_m[u]>0\right\}.$$
In order to compare these volumes we set up and solve a variational isoperimetric problem in Section 3. Somewhat unexpectedly this setting has carried out the Poincare type inequalities. These inequalities were first discovered by N. S. Trudinger and Xu-Jia Wang in
\cite{TW98}. They derived these inequalities by a different method but the most essential link is the same. Namely, it is the nontrivial solvability of the Dirichlet problem
\begin{equation}\label{w1} T_m[w]-T_l[w]=0,\quad w\arrowvert_{\partial\Omega}=0,\quad 0\leqslant l<m\leqslant n.\end{equation}
Equation (\ref{w1}) may be rewritten as $T_{m,l}[u]=1$ and in this form qualified as the simplest equation with  Hessian quotient operator
\begin{equation}T_{m,l}:=\frac{T_m[u]}{T_l[u]},\quad 1\leqslant l<m\leqslant n,\label{qua}\end{equation}
introduced in the papers of N.S.Trudinger
\cite{T90}, \cite{T95}.
Notice that a quotient operator $T_{m,l}[u]$ is not skew symmetric.
The sufficient close to necessary conditions for classic solvability of the Dirichlet problem to equation $T_{m,l}[u]=f>0$ had been found in the paper
\cite{T95}.
The following theorem is a particular case of Theorem 1.1 from this paper.
\begin{Th}
Let $\Omega\subset\mathbb R^n$ be a bounded domain, $\partial\Omega\in C^{4+\alpha}$. Assume that $\partial\Omega$ is $(m-1)$-convex. Then the problem (\ref{w1}) has the unique in $C^2(\Omega)$ nontrivial solution $w\in C^{4+\alpha}(\bar\Omega)$ for the odd $q=ml$ and two solution, $w$, $-w$, otherwise.
\end{Th}
A notion of the $p$-convexity of the hypersurface  via its $p$-curvature $\mathbf k_{p}[\partial\Omega]$ may be found in
\cite{Iv12}, \cite{IF14fix}
and an assumption from Theorem 1.1 is equivalent to the inequality $\mathbf k_{m-1}[\partial\Omega]>0$, $\mathbf k_{m-1}[\partial\Omega]$ is the $(m-1)$-curvature of $\partial\Omega$.

A brief of the theory of Hessian quotients $T_{m,l}$ is given in Section 4 of this paper.

In Section 5 we consider the direct approach to deduction of the Poincare type inequalities, what is to find an $m$-admissible minimizer to the functional
\begin{equation}\label{vw}J_{m,l}[u]:=\frac{I_m^{\frac{1}{m+1}}[u]}{I_l^{\frac{1}{l+1}}[u]},\quad u\arrowvert_{\partial\Omega}=0,\quad0\leqslant l<m\leqslant n.\end{equation}
The answer is known, Section 3. Namely, the unique nontrivial solution of the problem (\ref{w1}) with $w=w_{m,l}\leqslant 0$ provides minimum to the functional (\ref{vw}) on the set of $m$-admissible functions. Hence, $\delta^2J_{m,l}[w_{m,l}]\geqslant 0$ onto this set. The latter carries out a collection of regulated by functions $\{w_{m,l}\}$ new inequalities. The following proposition presents a sample of those.
\begin{Th}
Let $\partial\Omega\in C^{4+\alpha}$, $u\in{\stackrel{\circ}{W}}\vphantom {W}_1^2(\Omega)$. Assume that Gauss curvature of $\partial\Omega$ is positive. Then
\begin{equation}\label{P2}\frac{n-1}{\int_\Omega|w_x|^2dx}\left(\int_\Omega u\Delta wdx\right)^2+\int_\Omega|u_x|^2dx\leqslant\int_\Omega u_iu_j\frac{\partial}{w_{ij}}(\det w_{xx})dx,\end{equation}
where $w\leqslant0$ is the nontrivial solution to the problem (\ref{w1}) with $l=1$, $m=n$.
\end{Th}

\section{On the skew symmetry of fully nonlinear differential operators}
In order to indicate the idea of started in mid-seventies formalism (see, for instance,
\cite{Re73},\cite{Ru74}, \cite{I75}, \cite{B77}), \cite{Ly}, \cite{BT78}),
we present a slightly updated version of Theorem 2.1 from
\cite{I80}.
\begin{Th}
Let $\Omega\subset\mathbb R^n$ be a bounded domain, $v=(v^1,\dots,v^n)^T\in C^1(\bar\Omega)$:
$$v_i:=\frac{\partial v}{\partial x^i},\quad v_x:=(v^k_i)_1^n.$$
The following statements are equivalent:
\vskip .1in
($i$) Lagrangian $F[v]=F(v_x)$ belongs to the kernel of variational derivative, i.e.,  $\int_{\Omega}F(v_x)dx$ does not depend on $v(x)$, $x\in\Omega$;
\vskip .1in
($ii$) the identities
\begin{equation}\label{ann}\frac{\partial}{\partial x^i}\frac{\partial F[v]}{\partial v^k_i }\equiv0,\quad k=1,\dots,n\end{equation}
are valid;
\vskip .1in
($iii$) an operator $F[v]=F(v_x)$ is a linear combination of an arbitrary order minors of $\det v_x$.
\end{Th}
The skew symmetry of minors is of common knowledge and it has turned out that $(i)$, $(ii)$ are the consequences of this property via $(iii)$.

\begin{Def}
We say an operator $F[v]=F(v_x)$, $v=(v^1,\dots,v^n)^T\in C^1(\bar\Omega)$ is skew symmetric if it is a linear combination of an arbitrary order minors of $\det v_x$.
\end{Def}

Notice that it is rather senseless to speak about the skew symmetry, when only minors of the first order are taken in $(iii)$. In this case Theorem 2.1 is a triviality. Nevertheless, the divergence free linear differential operators might be qualified as generated by skew symmetric ones.

This amazing property had been a starting point to some important developments in quite different areas of mathematics and not surprisingly the choice of $v$ as well as notations were different therein. For instance, the authors of
\cite{Ru74}, \cite{B77}
had worked with vector-fields $v\in \mathbb R^n$. In the paper
\cite{Re73}
the vector-functions $v=u_x/\sqrt{1+u_x^2}$, $u\in C^2$, are under consideration  and geometric curvature operators have been investigated from this point of view.

In presented paper the case $v=u_x$ , i.e, Hessian operators, generated by Hesse matrix $u_{xx}$, is of main interest. The following proposition is well known since long and in order to underline its connection with the skew symmetry, we formulate it in our terminology.
\begin{Cor}
Let $v=u_x$, $u\in C^2$. Assume that the operator $F[u]=F(u_{xx})$ is $m$-homogeneous and skew symmetric. Then
\begin{equation}\label{Id}F[u]\equiv \frac{1}{m}\frac{\partial}{\partial x^i}\left(u_j\frac{\partial F[u]}{\partial u_{ij}}\right)\equiv \frac{1}{m}\frac{\partial^2}{\partial x^i\partial x^j}\left(u\frac{\partial F[u]}{\partial u_{ij}}\right).\end{equation}
\end{Cor}
The simplest example of $m$-homogeneous and skew symmetric operator is $m$-Hessian operator (\ref{mh}):
$$T_m[u]=T_m(u_{xx}).$$
Recall that by the symbol $T_m(u_{xx})$ we denote the $m$-trace of the matrix $u_{xx}$, that is the sum of all $m$-order principal minors, $T_0\equiv1$.

The skew symmetry of minors may be considered as an equivalent of the skew symmetry of exterior $n$-forms. Such approach to $m$-homogeneous fully nonlinear operators has been described, for instance, in the paper
\cite{I88}.
Namely, denote by $\omega_{m,n-m}[v]$ the exterior form
\begin{equation}\omega_{m,n-m}[v]=\sum_{(i_1<\dots <i_m)\atop(i_{m+1}<\dots<i_n)}\sigma(\mathbf i)dv^{i_1}\wedge\dots\wedge dv^{i_m}\wedge dx^{i_{m+1}}\wedge\dots\wedge dx^{i_n},\label{ef}\end{equation}
where $\sigma(\mathbf i)$  equals to $1$ either $-1$ depending on evenness of permutation $(i_1,\dots i_m, i_{m+1},\dots,i_n)$. Denote also $\omega_n(x)=\omega_{0,n}[v]$. The following proposition is the result of straightforward computing via (\ref{ef}), (\ref{ann}).
\begin{Th}
Let $v\in C^1$. Then
\begin{equation}\omega_{m,n-m}[v]=T_m(v_x)\omega_n(x),\quad0\leqslant m\leqslant n.\label{form}\end{equation}
\end{Th}

It looks reasonable to interpret the $m$-homogeneous skew symmetric operators as operator-densities of some measures in $\Omega$, what carries out the restriction $T_m(v_x)>0$. Let, for instance, in (\ref{form})
$$v=u_x\quad\Rightarrow\quad\omega_{m,n-m}[v]=T_m(u_{xx})\omega_n(x),$$
\begin{equation}v=\frac{u_x}{\sqrt{1+u_x^2}}\quad\Rightarrow\quad\omega_{m,n-m}[v]=\mathbf k_m[\Gamma(u)]\omega_n(x),\label{ohc}\end{equation}
where $\mathbf k_m[\Gamma(u)]$ is $m$-curvature of the graph of $u$ (see \cite{Iv12}, \cite{IF14fix}).
So, if one plans to deal with geometric measures in the sense (\ref{ohc}), it is necessary to require $\mathbf k_m[\Gamma(u)]>0$.

If $T_m[u](x)>0$, the $m$-Hessian operator $T_m[u]=T_m(u_{xx})$ , $x\in\bar\Omega$, may be interpreted as $m$-Hessian operator-density of some measure in $\Omega$. Possibly, this was a reason to introduce a notion of ``Hessian measures'' in
\cite{TW97} under similar circumstances.

In order to describe some properties of $\omega_{m,n-m}[u_x]$, we fix orientation by the requirement $\int_{\Omega}\omega_n(x)>0$, $\Omega$ is a bounded domain in $\mathbb R^n$. This agreement and above argumentation single out a functional set $\{u\in C^2(\bar\Omega):\; T_m[u]>0\}$, $1\leqslant m\leqslant n$. The following theorem (see, for instance, \cite{IF14fix}) indicates some complications with these sets.
\begin{Th}
Let $\Omega$ be a bounded domain in $\mathbb R^n$, $\partial\Omega\in C^k, k\geqslant 2$. Assume there is a point $x_0\in\partial\Omega$ such that $\mathbf k_{m-1}[\partial\Omega](x_0)=0$.
Then
\begin{equation}\label{ems}\{u\in C^2(\bar\Omega):\;u\arrowvert_{\partial\Omega}={\rm const},\; T_m[u]>0\}=\emptyset,\end{equation}
whatever $1<m\leqslant n$ had been.
\end{Th}
Notice that $m=1$ is excluded from Theorem 2.5 because $\mathbf k_0[\partial\Omega]=1$ by definition.
Relation (\ref{ems}) shows that in contrast to the linear elliptic equations the theory of $m$-Hessian operators, $m>1$, is non local matter.

On the other hand, Theorem 3 from the paper
\cite{CNS85}, page 264,
contains some positive information. In our notations a slightly modified version of this theorem reads as
\begin{Th}
Let $f\in C^{2+\alpha}(\bar\Omega)$, $\partial\Omega\in C^{4+\alpha}$, $0<\alpha<1$. Assume that $f>0$ in $\bar\Omega$, $\mathbf k_{m-1}[\partial\Omega]>0$. Then the Dirichlet problem
\begin{equation}T_m(u_{xx})=f,\quad u\arrowvert_{\partial\Omega}={\rm const},\quad 1\leqslant m\leqslant n,\label{H}\end{equation}
admits a solution $u\in C^{4+\alpha}(\bar\Omega)$. Moreover, if in (\ref{H}) ${\rm const}\neq 0$ or $m=2k-1$, $u$ is a unique in $C^2(\Omega)$. In the case $m=2k$ and $u$ is vanishing on $\partial\Omega$, there are two solutions in $C^2(\Omega)$ : $u$, $-u$.
\end{Th}
The further development is restricted to the functional sets, supported by Theorem 2.6:
\begin{equation}{\stackrel{\circ}{\mathbb K}}\vphantom{\mathbb K}_m(\bar\Omega)=\{u\in {\stackrel{\circ}{C}}\vphantom{C}^2(\bar\Omega):\; T_m[u]>0,\; u\leqslant0\},\quad 1\leqslant m\leqslant n,\label{K0}\end{equation}
which are  sub-cones of the well known by now cones of $m$-admissible in $\bar\Omega$ functions. They admit many equivalent definitions (see for inst.
\cite{IPY})
and are denoted by different symbols (compare
\cite{Iv83}, \cite{CNS85}, \cite{TW97}).
The constructive definition of the cone of $m$-admissible functions has been given in the paper
\cite{Iv83}
and in updated notations reads as
\begin{equation}\mathbb K_m(\bar\Omega)=\{u\in C^2(\bar\Omega):\; T_p[u]>0,\; p=1,2,\dots,m\},\quad 1\leqslant m\leqslant n.\label{adm}\end{equation}
We show that
$$\mathbb K_m(\bar\Omega)\cap\{u\arrowvert_{\partial\Omega}=0\}={\stackrel{\circ}{\mathbb K}}\vphantom{\mathbb K}_m(\bar\Omega).$$
If $u\in\mathbb K_m(\bar\Omega)$, then $u_{xx}$ can not be negative definite matrix in any point of $\Omega$. So $u$ has no maximums in $\Omega$ and the requirement $u\arrowvert_{\partial\Omega}=0$ provides $u\leqslant 0$ in $\bar\Omega$.
In order to prove the reverse implication we consider a matrix analog of cone (\ref{adm}). Denote by ${\rm Sym}(n)$ the space of symmetric $n\times n$-matrices:
\begin{equation}\label{Km}K_m=\{S\in {\rm Sym}(n): \;T_p(S)>0,\; p=1,2,\dots,m\},\quad 1\leqslant m\leqslant n.\end{equation}
Let $S_0$ be a positive definite matrix. It is well known that $K_m$ is a connected in ${\rm Sym}(n)$ component of the set $\{S:T_m(S)>0\}$, containing $S_0$ (see for instance \cite{G59}, \cite{IPY}, \cite{IF14fix}, \cite{IF13}). A function $u\in{\stackrel{\circ}{\mathbb K}}\vphantom{\mathbb K}_m(\bar\Omega)$ attains minimum (may be not strong) in $\Omega$. Hence, the connected set $\{u_{xx},x\in\bar\Omega\}$ contains a positive definite matrix and requirement $T_m(u_{xx})>0$, $x\in\bar\Omega$ provides $u_{xx}\in K_m$, $x\in\bar\Omega$, i. e. $u\in\mathbb K_m(\bar\Omega)$.

The matching of definitions (\ref{K0}), (\ref{adm}) demonstrates once again the nonlocal nature of $m$-admissible functions.

\section{On the variational problems I}
It is natural to associate with forms (\ref{ef}) the integrals
\begin{equation}\label{gint}\int_\Omega h(x)\omega_{p,n-p}[v],\quad\Omega\subset\mathbb R^n,\quad v=(v^1,\dots,v^n),\quad p=1,\dots,n,\end{equation}
and speak about some volumes generated by $v$ if $h(x)>0,x\in\Omega$. If $v=u_x$, $h=-u$,  integrals (\ref{gint}) may be written in the following form (see (\ref{sim})):
$$H_m[u]:=-\int_\Omega uS_m[D^2u]dx.$$
Functional $H_n[u]$ was introduced in the paper
\cite{Ts},
while the paper
\cite{W94}
covers all $0<m\leqslant n$ and functionals $H_m[u]$, $m=1,\dots,n$. Therein these functionals have been named Hessian integrals. Later on the ideas from this paper were developed by many authors. For instance, in
\cite{ChW}
Hessian integrals were applied to study some analogs of the problems from the theory of semi-linear elliptic equations. Some properties of Hessian integrals discovered in the paper
\cite{TW98}
are of particular interest in context of our further proceeding.

We consider Hessian integrals from some different point of view and to begin with write them out in our notations:
\begin{equation}I_p[u]:=\int_\Omega(-u)\omega_{p,n-p}[u_x]=\int_\Omega (-u)T_p[u]dx,\quad u\in\mathbb {\stackrel{\circ}{\mathbb K}}\vphantom{\mathbb K}_p(\bar\Omega),\label{mFl}\end{equation}
$ p=0,\dots,n$.
Our goal is to compare these functionals for different $p$ and we set up the isoperimetric problem: to  find $\underline u$, which minimizes $I_m[u]$ in $\mathbb {\stackrel{\circ}{\mathbb K}}\vphantom{\mathbb K}_m[\Omega]$ under condition $I_l[u]=1$, $0\leqslant l<m\leqslant n$. In other words, we are looking for $\underline u$ such that
\begin{equation}I_m[\underline u]\leqslant I_m[u],\quad \underline u, u\in {\stackrel{\circ}{\mathbb K}}\vphantom{\mathbb K}_m(\bar\Omega)\cap\{I_l[u]=1\},\quad 0\leqslant l<m\leqslant n.\label{Isp}\end{equation}
The correctness of setting (\ref{Isp}) confirms
\begin{lemma}
Let $u \in C^2(\Omega)\cap{\stackrel{\circ}{C}}\vphantom{C}^1(\bar\Omega)$. Assume $I_p[u]=1$. Then the first variation of the functional (\ref{mFl}) is nonzero on $u$.
\end{lemma}
\begin{proof}
Indeed, let $\tilde u=u+th$, where $h$ is an arbitrary function from $C^2(\Omega)\cap{\stackrel{\circ}{C}}\vphantom{C}^1(\bar\Omega)$, $t\in\mathbb R$. Then
$$\frac{d}{dt}I_p[\tilde u]=-\int_\Omega(hT_p[\tilde u]+\tilde uT_p^{ij}[\tilde u]h_{ij})dx,\quad T_p^{ij}[\tilde{u}]=\frac{\partial T_p[\tilde{u}]}{\partial\tilde{u}_{ij}},\quad 1\leqslant i,j\leqslant n.$$
It follows from integration by parts and (\ref{Id}) that
\begin{equation}\frac{d}{dt}I_p[\tilde u]=-(p+1)\int_\Omega hT_p[\tilde u]dx.\label{ani}\end{equation}
Assume that $$\delta I_p[u]=\frac{d}{dt}I_p[\tilde u]\arrowvert_{t=0}=0.$$
Then relation (\ref{ani}) is equivalent to $T_p[u]\equiv0$. But it contradicts to the assumption $I_p[u]=1$, what validates Lemma 3.1.
\end{proof}

Notice that the correctness of problem (\ref{Isp}) is a consequence of identity (\ref{Id}), i.e., of the skew symmetry of $m$-Hessian operators. Next, we expose a link between the isoperimetric problem (\ref{Isp}) and Hessian quotients (\ref{qua}).

\begin{Th}
Let $0\leqslant l< m\leqslant n$. Assume there is $w\in\mathbb {\stackrel{\circ}{\mathbb K}}\vphantom{\mathbb K}_m[\bar\Omega]$ such that
\begin{equation}T_{m,l}[w]:=\frac{T_m[w]}{T_l[w]}=1.\label{quo}\end{equation}
Then there exists $\underline u$ satisfying (\ref{Isp}) and
\begin{equation}I_m[u]\geqslant I_m[\underline u]=I_m^\frac{l-m}{l+1}[w],\quad u\in {\stackrel{\circ}{\mathbb K}}\vphantom{\mathbb K}_m(\bar\Omega)\cap\{I_l[u]=1\}.\label{isw}\end{equation}
\end{Th}
\begin{proof}
The problem (\ref{Isp}) is a classic isoperimetric variational problem. Due to Lemma 3.1 there exists Lagrange multiplier $\lambda$ such that a minimizer to the functional
$$\int_\Omega-u(T_m[u]-\lambda T_l[u])dx,\quad u\in{\stackrel{\circ}{\mathbb K}}\vphantom{\mathbb K}_m(\bar\Omega),$$
solves the problem (\ref{Isp}). Hence, we are looking for solutions to the Euler -- Lagrange equation
\begin{equation}(m+1)T_m[u]-(l+1)\lambda T_l[u]=0,\label{EL}\end{equation}
what follows from (\ref{ani}). Since only functions $u\in{\stackrel{\circ}{\mathbb K}}\vphantom{\mathbb K}_m(\bar\Omega)$ are of interest, a multiplier $\lambda$ has to be positive.
Denote
$$\mu^{m-l}=\frac{l+1}{m+1}\lambda.$$
Then equation (\ref{EL}) turns into $T_{m,l}[u]=\mu^{m-l}$.
The function $\underline u=\mu w$, where $w$ is a solution to (\ref{quo}), satisfies condition in (\ref{Isp}), and hence solves the problem (\ref{Isp}). Moreover, we have constructed the sharp estimate:
$$I_m[w]=I_l[w]=\frac{1}{\mu^{l+1}}I_l[\underline u]=\frac{1}{\mu^{l+1}}\quad \Rightarrow\quad \mu=I^{-\frac{1}{l+1}}[w],$$
$$I_m[u]\geqslant I_m[\underline u]=\mu^{m+1}I_m[w]=I_m^\frac{l-m}{l+1}[w],\quad u\in{\stackrel{\circ}{\mathbb K}}\vphantom{\mathbb K}_m(\bar\Omega)\cap\{I_m[u]=1\}.$$
\end{proof}
An auxiliary Dirichlet problem (\ref{quo}) has appeared in the paper
\cite{TW98}
as a crucial tool to derive Poincare type inequalities for functionals $I_m[u]$, $1\leqslant m\leqslant n$, interpreted in a weak sense. For $u\in C^2(\bar\Omega)$ these inequalities spring up as a simple consequence of (\ref{isw}) and we write out their equivalents in
\begin{Cor}
Let $0\leqslant l\leqslant m\leqslant n$ and $w$ satisfies the equation (\ref{quo}).
Then
\begin{equation}\left(\frac{I_m[u]}{I_m[w]}\right)^\frac{1}{m+1}\geqslant\left(\frac{I_l[u]}{I_l[w]}\right)^\frac{1}{l+1},\quad u\in{\stackrel{\circ}{\mathbb K}}\vphantom{\mathbb K}_m(\bar\Omega).
\label{Poi}\end{equation}
\end{Cor}
\begin{proof}
Indeed, defined by the line $u=I_l^\frac{1}{l+1}[u]\tilde u$ function $\tilde u$ belongs to ${\stackrel{\circ}{\mathbb K}}\vphantom{\mathbb K}_m(\bar\Omega)\cap\{I_l[u]=1\}$. It follows from (\ref{isw}) that
\begin{equation}I_m^\frac{1}{m+1}[u]=I_l^\frac{1}{l+1}[u]I_m^\frac{1}{m+1}[\tilde u]\geqslant I_l^\frac{1}{l+1}[u]I_m^{\frac{l-m}{(l+1)(m+1)}}[w]=I_l^\frac{1}{l+1}[u]I_m^{\frac{1}{m+1}-\frac{1}{l+1}}[w].\label{ss}\end{equation}
\end{proof}
Notice that inequality (\ref{Poi}) is a symmetrized form of restricted to $u\in{\stackrel{\circ}{\mathbb K}}\vphantom{\mathbb K}_m(\bar\Omega)$ inequality (1.13) from
\cite{TW98}.
Also a solution $w_\mu\in{\stackrel{\circ}{\mathbb K}}\vphantom{\mathbb K}_m(\bar\Omega)$ to equation $T_{m,l}[w]=\mu^2$ with an arbitrary $\mu\in\mathbb R^+$ may be taken in capacity of $w$ in relation (\ref{Poi}).

Properties of solutions to equation (\ref{quo}) from ${\stackrel{\circ}{\mathbb K}}\vphantom{\mathbb K}_m(\bar\Omega)$ are of our special interest and the first one we present as a consequence of the sharp inequalities (\ref{isw}).
\begin{Th}
Let $0\leqslant l\leqslant p<m$. Assume that for every $p$  there is a solution $w_{m,p}\in{\stackrel{\circ}{\mathbb K}}\vphantom{\mathbb K}_m(\bar\Omega)$ to equations $T_{m,p}[w_{m,p}]=1$. Then
\begin{equation}\label{Iw}I_m^{m-l}[w_{m,l}]\geqslant I_m^{\frac{l+1}{p+1}(m-p)}[w_{m,p}]I_l^{\frac{m+1}{p+1}(p-l)}[w_{p,l}]. \end{equation}
\end{Th}
\vskip .1in
\begin{proof}
We use the inequality (\ref{Poi}) in the form (\ref{ss}):
$$I_m^\frac{1}{m+1}[u]\geqslant c_{m,l}I_l^\frac{1}{l+1}[u],\quad c_{m,l}=I_m^{\frac{1}{m+1}-\frac{1}{l+1}}[w_{m,l}],\quad u\in{\stackrel{\circ}{\mathbb K}}\vphantom{\mathbb K}_m(\bar\Omega).$$
A constant $c_{m,l}$ is sharp, because the above inequality turns into equality, when $u=w_{m,l}$. Using the inequality (\ref{ss}) twice, we derive
$$I_m^{\frac{1}{m+1}}[u]\geqslant c_{m,p}c_{p,l}I_l^{\frac{1}{l+1}}[u],\quad u\in{\stackrel{\circ}{\mathbb K}}\vphantom{\mathbb K}_m(\bar\Omega),$$
where a constant $c_{m,p}c_{p,l}$ is not sharp. Hence, $c_{m,l}\geqslant c_{m,p}c_{p,l}$, what coincides with relation (\ref{Iw}).
\end{proof}

\section{Some properties of the Hessian quotients}
To make Theorem 3.2 credible it is necessary to confirm the solvability of the problem (\ref{quo}) and we present some extraction from general theory.
The existence of admissible solutions to the Dirichlet problem for the Hessian quotient equations was proved in the paper
\cite{T95},
Theorem 1.1, p.153 and in the author's notations it reads as
\begin{Th}
Let $0\leqslant l<m\leqslant n$ and $\Omega$ be a bounded uniformly $(m-1)$-convex domain in $\mathbb R^n$, with $\partial\Omega\in C^{3,1}$, $\varphi\in C^{3,1}(\partial\Omega)$ and let $\psi$ be a positive function in $C^{1,1}(\bar\Omega)$. Then the Dirichlet problem,
\begin{equation}F(D^2u)=S_{m,l}(\lambda[D^2u])=\psi\quad in\quad \Omega,\quad u=\varphi\quad on\quad \partial\Omega,\label{Trq}\end{equation}
is uniquely solvable for admissible $u\in C^{3,\alpha}(\bar\Omega)$ for any $0<\alpha<1$.
\end{Th}
It is more than 20 years since this amazing theorem has been proved and now we suggest to slightly update its formulation. Namely,
\vskip .1in
$(i)$ the basis of Theorem 4.1 is a construction of a priori estimates of solutions at the boundary and the requirement ``uniformly $(m-1)$-convex domain'' is equivalent to the inequality $\mathbf k_{m-1}[\partial\Omega]>0$, what means that the hyper-surface $\partial\Omega$ is $(m-1)$-convex. The definition of $\mathbf k_{m-1}$-curvature of the hyper-surface $\partial\Omega$ and reasons for such substitute may be found in
\cite{Iv12}, \cite{IF14fix};
\vskip .1in
$(ii)$ in our argument we do not allude to the eigenvalues $\lambda[D^2u]$ and write the equation in (\ref{Trq}) as $T_{m,l}[u]=\psi$ (see (\ref{sim}), (\ref{qua})), what allows to differentiate our equations, when necessary, without preliminary passes;
\vskip .1in
$(iii)$ the assertion of Theorem 4.1 is equivalent to ``There exists the unique in  $\mathbb K_m(\bar\Omega)$ solution $u$ to the problem (\ref{Trq}) and $u\in C^{3+\alpha}(\bar\Omega)$ for any $0<\alpha<1$''.
\vskip .1in
Notice that if $\varphi=0$, a unique solution from Theorem 4.1 belongs to ${\stackrel{\circ}{\mathbb K}}\vphantom{\mathbb K}_m(\bar\Omega)$ (see description of the cones (\ref{K0}), (\ref{adm})), what means that it is unique in
\begin{equation}\label{C-}{\stackrel{\circ}{C}}\vphantom{C}^2_-(\bar\Omega):=\{u\in {\stackrel{\circ}{C}}\vphantom{C}^2(\bar\Omega):\; u\leqslant 0\}.\end{equation}
More precisely, the following consequence of Theorem 2.5 and properties of the cones (\ref{K0}) -- (\ref{Km}) is valid.

\begin{lemma}
Let $0\leqslant l<m\leqslant n$, $\partial\Omega\in C^2$. There are two possibilities:
\vskip .1in
$(i)$ if there exists $x_0\in\partial\Omega$ such that $\mathbf k_{m-1}[\partial\Omega](x_0)=0$, then
$$\{u\in {\stackrel{\circ}{C}}\vphantom{C}^2(\bar\Omega):\;T_{m,l}[u]>0\}=\emptyset;$$

$(ii)$ if $x_0$ from $(i)$ does not exist, then
$$\{u\in {\stackrel{\circ}{C}}\vphantom{C}^2_-(\bar\Omega):\;T_{m,l}[u]>0\}={\stackrel{\circ}{\mathbb K}}\vphantom{\mathbb K}_m(\bar\Omega).$$
\end{lemma}

It follows from Theorem 4.1, Lemma 4.2 that the cone ${\stackrel{\circ}{\mathbb K}}\vphantom{\mathbb K}_m(\bar\Omega)$ is a natural set of solvability of the problem
\begin{equation} T_{m,l}[u]=\psi>0,\quad u\arrowvert_{\partial\Omega}=0\label{DqT}\end{equation}
and the requirement of $(m-1)$-convexity of $\partial\Omega$ is necessary. Notice that the half-space (\ref{C-}) was introduced to avoid speciality of even values of the number $m+l$. Similar to situation in Theorem 2.6, in this case the inequality $T_{m,l}[u]>0$ carries out two cones in ${\stackrel{\circ}{C}}\vphantom{C}^2(\bar\Omega)$.

A correct setting of the Dirichlet problem (\ref{Trq}) assumes that operator $F$ is elliptic on the set of admissible functions, what has been proved in the paper
\cite{T90}
by combinatoric methods. We offer somewhat different approach and consider the ellipticity of $F$ as a consequence of the positive monotonicity of operators $T_{m,l}[u]$ in $\mathbb K_m(\bar\Omega)$.

To begin with we consider a set of functions $\{T_p=T_p(S)\}_1^n$ in the matrix cone (\ref{Km}) and denote
\begin{equation}\label{shT}T^{ij}_p(S):=\frac{\partial T_p}{\partial s_{ij}}(S),\quad 1\leqslant i,j\leqslant n.\end{equation}
Notice that
$$T_{m-1;i}(S):=\frac{\partial T_{m}}{\partial s_{ii}}(S)$$
is the $(m-1)$-trace of the matrix $S$ with deleted $i$-th row and column. It is known that
\begin{equation}T_{m-1;i}(S)>0,\quad S\in K_m, \quad m=1,\dots,n.\label{der}\end{equation}
In this course we associate with the quotient operator $T_{m,l}[u]$ a functional quotient
\begin{equation}T_{m,l}(S):=\frac{T_m(S)}{T_l(S)},\quad 0\leqslant l< m\leqslant n,\quad S\in {\rm Sym}(n),\label{Squo}\end{equation}
and prove its monotonicity in the matrix cone $K_m$.

Denote by ${\rm Sym}^+(n)\subset{\rm Sym}(n)$ the set of positive definite matrices.

\begin{Th}
Let $S^0\in \overline{Sym^+}(n)$, $0\leqslant l<m\leqslant n$. Assume that $S^0\ne{\mathbf 0}$. Then
\begin{equation}T_{m,l}(S+S^0)>T_{m,l}(S),\quad S\in K_m.\label{mon}\end{equation}
\end{Th}

\begin{proof}
The proof consists of three steps.
\vskip .1in
I. We fix a matrix $S\in K_m$, an index $1\leqslant i\leqslant n$ and associate with them an auxiliary matrix:
$$S(t;i)=(s_{kl}+t\delta_{ki}\delta_{li})_1^n,\quad t\in \mathbb R.$$
When $l=0$, $T_{m,0}=T_m$ and due to (\ref{der}) we have $$T_m(S(t;i))=T_m(S)+tT_{m-1;i}(S)>T_m(S),\quad t>0.$$
Let
$$\underline t:=-\frac{T_m(S)}{T_{m-1;i}(S)}.$$
Then $T_m(S(\underline{t};i))=0$.
Moreover, $S(t,i)\in K_m$ for all $t>\underline{t}$ because the cone $K_m$ is a connected component of the set $\{S:T_m(S)>0\}$.

For the case $l>0$ we introduce an auxiliary function:
\begin{equation}y(t)=T_{m,l}(S(t;i)),\quad t\in\mathbb R.\label{yS}\end{equation}
Due to the Maclaurin inequality (see for instance \cite{Iv83}, \cite{IF13})
\begin{equation}\left(\frac{T_l(S)}{C_n^l}\right)^{\frac{1}{l}}\geqslant \left(\frac{T_m(S)}{C_n^m}\right)^{\frac{1}{m}},\quad S\in K_m,\label{mac}\end{equation}
we have an estimate
$$y(t)\leqslant c(m,n)\left(T_m(S(t;i))\right)^{1-\frac{l}{m}},\quad t>\underline{t}.$$
Hence, $y(t)\rightarrow 0$ when $t\rightarrow\underline{t}$.
\vskip .1in
II. The differentiation of function (\ref{yS}) brings out the lines
$$y^\prime(t)=\frac{T_{l-1;i}(S)}{T_l(S(t;i))}\left(\frac{T_{m-1;i}(S)}{T_{l-1;i}(S)}-y(t)\right),$$ \begin{equation}y^{\prime\prime}(t)=-2\frac{T_{l-1;i}(S)}{T_l(S(t;i))}y^\prime(t).\label{yy}\end{equation}
Integrating the ODE in (\ref{yy}) we derive $$y^\prime(t)=y^\prime(t_0)\frac{T^2_l(S(t_0;i))}{T^2_l(S(t;i))}.$$
Consider the initial value $y^\prime(t_0)$. It follows from I. and (\ref{der}) that there exists $t_0$ such that $\underline{t}<t_0<0$ and $y^\prime(t_0)>0$. Hence $y^\prime(t)>0$ for $t\geqslant t_0$ and we have arrived to the inequality
$$T_{m,l}(S)<T_{m,l}(S(t;i))<\frac{T_{m-1;i}(S)}{T_{l-1;i}(S)}=
\lim_{t\to +\infty}T_{m,l}(S(t;i)),\quad t>0,\quad i=1,\dots,n.$$
\vskip .1in
III. Consider first a diagonal matrix $S_d^0\in\overline{{\rm Sym}^+(n)}$, $S_d^0\ne{\mathbf 0}$. The inequality (\ref{mon}) with  $S^0=S_d^0$ follows from II. Since $p$-traces are orthogonal invariant (see (\ref{ort})), the inequality (\ref{mon}) is also true for an arbitrary nonzero matrix $S_0\in\overline{{\rm Sym}^+(n)}$.
\end{proof}
The inequality
\begin{equation}\label{elS}(T_{m,l}^{ij}(S)\xi,\xi)>0,\quad S\in K_m,\quad\xi\in\mathbb R^n,\quad |\xi|=1,\end{equation}
is a straightforward consequence of monotonicity (\ref{mon}). An operator version of (\ref{elS}) reads as
\begin{equation}(T_m^{ij}-T_{m,l}T_l^{ij})[u]\xi_i\xi_j>0,\quad T_p^{ij}[u]=\frac{\partial T_p(u_{xx})}{\partial u_{ij}},\quad u\in\mathbb K_m(\bar\Omega), \label{elp}\end{equation}
what means that operator quotients $T_{m,l}[u]$ are elliptic onto $\mathbb K_m(\bar\Omega)$.

So, equation (\ref{quo}) is uniquely solvable in $\mathbb K_m^0(\bar\Omega)$ due to Theorem 4.1 and the following consequence is of the principal interest in our paper.
\begin{Th}
Let $\Omega$ be a bounded domain in $\mathbb R^n$, $\partial\Omega\in C^{4+\alpha}$, $0<\alpha<1$. Assume $\mathbf k_{m-1}[\partial\Omega]>0$. Then there exists the unique in $C^2_{-0}(\bar\Omega)$ solution $w$ to the problem
\begin{equation}T_{m,l}[w]=1,\quad w\arrowvert_{\partial\Omega}=0,\quad0\leqslant l<m\leqslant n.\label{0quo}\end{equation}
Moreover, $w\in\mathbb K_m(\bar\Omega)\cap C^{4+\alpha}(\bar\Omega)$ and satisfies the inequality
\begin{equation}(T_m^{ij}-T_l^{ij})[w]\xi_i\xi_j>0,\quad |\xi|=1.\label{wel}\end{equation}
\end{Th}
Notice that an existence part of Theorem 4.4 is identical to Theorem 1.1, while inequality (\ref{wel}) coincides with ellipticity condition (\ref{elp}) with $u=w$.

\begin{remark} Quotients operators $T_{m,l}[u]$ are not divergence free, when $l\geqslant1$. It means that skew symmetry does not matter for solvability of the Dirichlet problem for Hessian equations. However, equation (\ref{0quo}) may be written as $(T_m-T_l)[w]=0$ in $\mathbb K_m(\bar\Omega)$. Due to identities (\ref{Id}) the latter is equivalent to
$$
\frac{\partial}{\partial x^i}A^{ij}[w]w_j=0,\quad A^{ij}[w]=\left(\frac{1}{m}T_m^{ij}-\frac{1}{l}T_l^{ij}\right)[w].
$$
\end{remark}
For the fixed $1<m\leqslant n$ we consider now the set of solutions
$$\{w_{m,l}, l=0\dots,m-1\}$$
from Theorem 4.4. It is natural to await some connections between these functions. At the moment we know the following.
\begin{lemma}
Under conditions of Theorem 4.4 the inequalities
\begin{equation}\label{w2}T_p[w_{m,l}]>1, \quad w_{m,l}<w_{m,0},\quad x\in\Omega,\end{equation}
hold true for all $1\leqslant l,p\leqslant m-1$.
\end{lemma}
\begin{proof}
To prove the lefthand side of (\ref{w2}) we apply a strong version of the Maclaurin inequality (\ref{mac}):
\begin{equation}{\label{ML}}T_m^{\frac{1}{m}}[w]<T_l^{\frac{1}{l}}[w],\quad 1\leqslant l\leqslant m-1,\quad w\in\mathbb K_m(\bar\Omega).\end{equation}
Let $w=w_{m,l}$. By definition $T_{m,l}[w_{m,l}]=1$ and due to (\ref{ML}) we have
\begin{equation}\label{w3}1=\frac{T_m[w_{m,l}]}{T_l[w_{m,l}]}<T_m^\frac{m-l}{m}[w_{m,l}]
<T_p^\frac{m-l}{p}[w_{m,l}].\end{equation}
The second part of (\ref{w2}) is a consequence of well known comparison theorem for $m$-Hessian operators. Indeed, it follows from (\ref{w3}) that $T_m[w_{m,l}]>1$. On the other hand, $T_m[w_{m,0}]=1$ by definition. Via the comparison principle the inequality for $m$-Hessian operators guarantees the reverse inequality for functions from ${\stackrel{\circ}{\mathbb K}}\vphantom{\mathbb K}_m(\bar\Omega)$, i.e., the second inequality in (\ref{w2})
\end{proof}

\section{On the variational problem II}
Theorems 3.2, 4.4 carry out
\begin{Th}
Let $\Omega$ be a bounded domain in $\mathbb R^n$, $\partial\Omega\in C^{4+\alpha}$, $0\leqslant l<m\leqslant n$. Assume $\mathbf k_{m-1}[\partial\Omega]>0$. Then there is a sharp constant $\mathbf c=c(l,m,\mathbf k_{m-1}[\partial\Omega])>0$ such that
\begin{equation}J_{m,l}[u]:=\frac{\left(\int_\Omega-uT_m[u]dx\right)^{\frac{1}{m+1}}}{\left(\int_\Omega-uT_l[u]dx\right)^{\frac{1}{l+1}}}
\geqslant\mathbf c,\quad u\in{\stackrel{\circ}{\mathbb K}}\vphantom{\mathbb K}_m(\bar\Omega),\label{Rt}\end{equation}
\end{Th}
Indeed, due to assumption $\mathbf k_{m-1}[\partial\Omega]>0$ there exists the unique in $C^2(\bar\Omega)$ solution $w=w_{l,m}\in\mathbb {\stackrel{\circ}{\mathbb K}}\vphantom{\mathbb K}_m(\bar\Omega)$ to the problem (\ref{0quo}). Therefore the inequality (\ref{Rt}) with $\mathbf c=J_{m,l}[w_{l,m}]$ is a replica of (\ref{Poi}).
\vskip .1in
Notice that inequalities (\ref{Rt}) are equivalent to the Poincare-type inequalities from the paper
\cite{TW98}.
If the principal goal of our paper had been to give a straightforward proof of those, it would be reasonable to set up a classic variational problem of minimization of the functional  $J_{m,l}[u]$ over the cone ${\stackrel{\circ}{\mathbb K}}\vphantom{\mathbb K}_m(\bar\Omega)$. In order to produce some new analogs of the classic Poincare inequality we outline this approach.
\begin{Th}
Assume conditions of Theorem 5.1 are satisfied and let $u$ be from the Sobolev space
${\stackrel{\circ}{W}}\vphantom{W}^2_1(\Omega)$, $w$ be a solution to the problem (\ref{0quo}). Then
\begin{equation}\frac{m-l}{I_m[w]}\left(\int_\Omega uT_m[w]dx\right)^2+\int_\Omega T_l^{ij}[w]u_iu_jdx\leqslant\int_\Omega T_m^{ij}[w]u_iu_jdx.\label{anpo}\end{equation}
\end{Th}
\begin{proof}
Let $\tilde w= w+th$, $t\in\mathbb R$, $h\in {\stackrel{\circ}{C}}\vphantom{C}^2(\bar\Omega)$. Similar to (\ref{ani}) we derive
\begin{equation}\frac{d^2}{dt^2}I_p[w]\equiv(p+1)\int_\Omega T_p^{ij}[\tilde{w}]h_ih_jdx,\quad p=1,\dots,n.\label{2ani}\end{equation}
It follows from (\ref{Rt}) that $w$ minimizes $J_{m,l}[u]$ over ${\stackrel{\circ}{\mathbb K}}\vphantom{\mathbb K}_m(\bar\Omega)$ and hence $\delta J_{m,l}[w]=0$, $\delta^2 J_{m,l}[w]\geqslant0$. Keeping in mind that $T_m[w]=T_l[w]$, $I_m[w]=I_l[w]$, we compute via (\ref{ani}), (\ref{2ani}) the second variation of functional $J_{m,l}[\tilde w]$:
\begin{equation}\delta^2J_{m,l}[w]=\frac{J_{m,l}[w]}{I_m[w]}\left(\frac{l-m}{I_m[w]}\left(\int_\Omega hT_m[w]dx\right)^2+\int_\Omega(T_m^{ij}-T_l^{ij})[w]h_ih_jdx\right).\label{v2}\end{equation}
Since the case $t=0$ is of interest, we may without loss of generality assume that $\tilde w\in {\stackrel{\circ}{\mathbb K}}\vphantom{\mathbb K}_m(\Omega)$ for an arbitrary $h\in C^2(\bar\Omega)\cap {\stackrel{\circ}{C}}\vphantom{C}_m^1(\bar\Omega)$. Therefore, relation (\ref{v2}) and a choice of $w$ provide $\delta^2 J_{m,l}[w]\geqslant0$, hence inequality (\ref{anpo}) is valid for an arbitrary function $u=h\in C^2(\bar\Omega)$. The case of $u\in {\stackrel{\circ}{W}}\vphantom{W}^2_1(\Omega)$ may be derived by approximation.
\end{proof}
Letting $l=1$, $m=n$ in Theorem 5.2 one sees exactly Theorem 1.2.
The case $l=0$ in Theorem 5.2 is of special interest and we extract it as
\begin{Cor}
Let $u\in{\stackrel{\circ}{W}}\vphantom{W}^2_1(\Omega)$ be an arbitrary function, $w_m\in C^2(\bar\Omega)$ a solution to the problem $T_m[w_m]=1$, $w_m\arrowvert_{\partial\Omega}=0$, $w_m\leqslant 0$. Then the inequalities
\begin{equation}m\left(\int_\Omega udx\right)^2\leqslant\int_\Omega -w_mdx\int_\Omega T_m^{ij}[w_m]u_iu_jdx,\quad m=1,\dots,n,\label{0l}\end{equation}
are true.
\end{Cor}
Notice that Corollary 5.3 implicitly contains the requirement of $(m-1)$-convexity of $\partial\Omega$.

The inequality (\ref{0l}) with $m=1$ and under requirement $\Delta u>0$ in a weak sense was attributed as Poincare inequality in the paper
\cite{TW98}.
Theorem 5.2 along with Corollary 5.3 is valid for an arbitrary function $u$ from  ${\stackrel{\circ}{W}}\vphantom{W}^2_1(\Omega)$ and speaking formally inequality (\ref{0l}) with $m=1$ is more general than its analog from
\cite{TW98}.

All the inequalities (\ref{anpo}) are sharp and the set (\ref{0l}) might be considered as a set of depending on the $p$-convexity of $\partial\Omega$ analogs to the classic Poincare inequality.
\begin{remark}
There are two questions concerning our inequalities:
\vskip .1in
I. Assume that $\mathbf k_{n-1}[\partial\Omega]>0$ in Corollary 5.3. Then we have a set of functions $\{w_m\}_1^n$ and relevant sharp inequalities (\ref{0l}). Is it possible to compare them for different values of $m$?
\vskip .1in
II. Let $m>1$ be fixed and assumptions of Theorem 5.2 satisfied. Then we have a collection of functions $\{w_{l,m}\}_0^{m-1}$. Are they comparable?
\end{remark}
Eventually we rewrite general inequality (\ref{anpo}) in the invariant under dilation form. Denote
$$
<u,v>_p= \int_\Omega T_p^{ij}[w]u_iv_jdx,\quad p=1,\dots,n,
$$
and let $w$ be a solution to the problem $T_{m,l}[w]=\mu$, $w\arrowvert_{\partial\Omega}=0$,  $u\in{\stackrel{\circ}{W}}\vphantom{W}^2_1(\Omega)$ . Then the inequality
\begin{equation}\label{inv}(m-l)\frac{<u,w>_l}{<w,w>_l}\frac{<u,w>_m}{<w,w>_m}\leqslant m\frac{<u,u>_m}{<w,w>_m}-l\frac{<u,u>_l}{<w,w>_l}\end{equation}
is equivalent to (\ref{anpo}), whatever $\mu\in\mathbb R^+$ has been. It follows from (\ref{inv}) that the constant $\mathbf c$ in (\ref{Rt}) is invariant under dilation.

\end{document}